\documentclass[oneside,a4paper,12pt,reqno]{amsproc}
\usepackage{amssymb}
\usepackage[english]{babel}
\usepackage[T1]{fontenc}
\usepackage[latin1]{inputenc}
\usepackage{fullpage}

\newcommand{\Ham}{\mathrm{Ham}}
\newcommand{\Lee}{\mathrm{Lee}}

\DeclareMathOperator{\GL}{GL}

\DeclareMathOperator{\cir}{circ}

\theoremstyle{plain}

\newtheorem{lma}{Lemma}[section]

\theoremstyle{definition}
\newtheorem{dfn}{Definition}[section]

\usepackage{amsfonts,latexsym,graphicx}
\usepackage{hyperref}
\usepackage{times}
\usepackage{enumerate}

\newcommand{\F}{\mathbb{F}}
\newcommand{\N}{\mathbb{N}}
\newcommand{\Z}{\mathbb{Z}}

\providecommand{\abs}[1]{\lvert#1\rvert}

\begin{document}

\title{Double and bordered $\alpha$-circulant self-dual codes over finite commutative chain rings}

\author{Michael Kiermaier}
\address{Michael Kiermaier\\Mathematical Department\\University of Bayreuth\\D-95440 Bayreuth\\Germany}
\email{michael.kiermaier@uni-bayreuth.de}
\urladdr{http://www.mathe2.uni-bayreuth.de/michaelk/}

\author{Alfred Wassermann}
\address{Alfred Wassermann\\Mathematical Department\\University of Bayreuth\\D-95440 Bayreuth\\Germany}
\email{alfred.wassermann@uni-bayreuth.de}
\urladdr{http://did.mat.uni-bayreuth.de/~alfred/home/index.html}


\begin{abstract}
In this paper we investigate codes over finite commutative rings $R$, whose generator matrices are built from $\alpha$-circulant matrices.
For a non-trivial ideal $I<R$ we give a method to lift such codes over $R/I$ to codes over $R$, such that some isomorphic copies are avoided.

For the case where $I$ is the minimal ideal of a finite chain ring we refine this lifting method:
We impose the additional restriction that lifting preserves self-duality. It will be shown that this can be achieved by solving a linear system of equations over a finite field.

Finally we apply this technique to $\Z_4$-linear double nega-circulant and bordered circulant self-dual codes. We determine the best minimum Lee distance of these codes up to length $64$.
\end{abstract}

\maketitle

\section{$\alpha$-circulant matrices}
\label{section:circulant}
In this section, we give some basic facts on $\alpha$-circulant matrices, compare with \cite[chapter~16]{MacWilliams-Sloane-1977-The_Theory_of_Error_Correcting_Codes}, where some theory of circulant matrices is given, and with \cite[page~84]{Davis-2004-Circulant_Matrices}, where $\alpha$-circulant matrices are called $\{k\}$-circulant.

\begin{dfn}
Let $R$ be a commutative ring, $k$ a natural number and $\alpha\in R$.
A $(k\times k)$-matrix $A$ is called \emph{$\alpha$-circulant}, if $A$ has the form
\[
\begin{pmatrix}
a_0            & a_1            & a_2        & \hdots & a_{k-2}        & a_{k-1} \\
\alpha a_{k-1} & a_0            & a_1        & \hdots & a_{k-3}        & a_{k-2} \\
\alpha a_{k-2} & \alpha a_{k-1} & a_0        & \hdots & a_{k-4}        & a_{k-3} \\
\vdots         & \vdots         & \vdots     &        & \vdots         & \vdots  \\
\alpha a_1     & \alpha a_2     & \alpha a_3 & \hdots & \alpha a_{k-1} & a_0
\end{pmatrix}
\]
with $a_i \in R$ for $i \in \{0,\ldots,k-1\}$.
For $\alpha = 1$, $A$ is called \emph{circulant}, for $\alpha = -1$, $A$ is called \emph{nega-circulant} or \emph{skew-circulant}, and for $\alpha = 0$, $A$ is called \emph{semi-circulant}.

An $\alpha$-circulant matrix $A$ is completely determined by its first row $v = (a_0, a_1,\ldots,a_{k-1})\in R^k$. We denote $A$ by $\cir_\alpha(v)$ and say that $A$ is the $\alpha$-circulant matrix \emph{generated} by $v$.
\end{dfn}
In the following, $\alpha$ usually will be a unit or even $\alpha^2 = 1$.

We define $T_{\alpha}=\cir_\alpha(0,1,0,\ldots,0)$, that is
\[
T_\alpha = \begin{pmatrix}
       & 1 &   &        &   \\
       &   & 1 &        &   \\
       &   &   & \ddots &   \\
       &   &   &        & 1 \\
\alpha &   &   &        & 
\end{pmatrix}
\]

Using $T_\alpha$, there is another characterization of an $\alpha$-circulant matrix:
A matrix $A\in R^{k\times k}$ is $\alpha$-circulant iff $A T_\alpha = T_\alpha A$.
This is seen directly by comparing the components of the two matrix products.

In the following it will be useful to identify the generating vectors $(a_0, a_1, \ldots, a_{k-1})\in R^n$ with the polynomials $\sum_{i=0}^{k-1}a_i x^i \in R[x]$ of degree at most $k-1$, which again can be seen as a set of representatives of the $R$-algebra $R[x]/(x^k - \alpha)$. Thus, we get an injective mapping $\cir_\alpha: R[x]/(x^k - \alpha)\rightarrow R^{k\times k}$.

Obviously $\cir_\alpha(1) = I_k$, which denotes the $(k\times k)$-unit matrix, $\cir_\alpha(\lambda f) = \lambda\cir_\alpha(f)$ and $\cir_\alpha(f + g) = \cir_\alpha(f) + \cir_\alpha(g)$ for all scalars $\lambda\in R$ and all $f$ and $g$ in $R[x]/(x^k - \alpha)$.
Furthermore, it holds $\cir_\alpha(e_i) = \cir_\alpha(x^i) = T_\alpha^i$ for all $i\in\{0,\ldots,k-1\}$ and $\cir_\alpha(x^k) = \cir_\alpha(\alpha) =  \alpha I_k = T_\alpha^k$, where $e_i$ denotes the $i$th\footnote{Throughout this article, counting starts at 0. Accordingly, $\N = \{0,1,2,\ldots\}$} unit vector.
So we have $\cir_\alpha(x^i x^j) = \cir_\alpha(x^i)\cir_\alpha(x^j)$ for all $\{i,j\}\subset\N$.
By linear extension it follows that $\cir_\alpha$ is a monomorphism of $R$-algebras.
Hence the image of $\cir_\alpha$, which is the set of the $\alpha$-circulant $(k\times k)$-matrices over $R$, forms a commutative subalgebra of the $R$-algebra $R^{k\times k}$ and it is isomorphic to the $R$-algebra $R[x]/(x^k - \alpha)$. Especially, we get $\cir_\alpha(a_0,\ldots,a_{k-1}) = \sum_{i=0}^{k-1} a_i T_\alpha^i$.

\section{Double $\alpha$-circulant and bordered $\alpha$-circulant codes}
\label{section:circ_code}
\begin{dfn}
Let $R$ be a commutative ring and $\alpha\in R$. Let $A$ be an $\alpha$-circulant matrix.
A code generated by a generator matrix
\[
(I_k \mid A)
\]
is called \emph{double $\alpha$-circulant code}.
A code generated by a generator matrix
\[
\begin{pmatrix}I_k & \begin{array}{cc}\beta & \gamma\cdots\gamma\\ \begin{array}{c}\delta\\\vdots\\\delta\end{array} & A\end{array}\end{pmatrix}
\]
with $\{\beta,\gamma,\delta\}\subset R\}$ is called \emph{bordered $\alpha$-circulant code}.
The number of rows of such a generator matrix is denoted by $k$, and the number of columns is denoted by $n = 2k$.
\end{dfn}

As usual, two codes $C_1$ and $C_2$ are called \emph{equivalent} or \emph{isomorphic}, if there is a monomial transformation that maps $C_1$ to $C_2$.

\begin{dfn}
Let $R$ be a commutative ring and $k\in\N$.
The symmetric group over the set $\{0,\ldots,k-1\}$ is denoted by $S_k$.
For a permutation $\sigma\in S_k$ the permutation matrix $S(\sigma)$ is defined as $S_{ij} = \delta_{i,\sigma(j)}$, where $\delta$ is the Kronecker delta.
An invertible matrix $M\in\GL(k,R)$ is called \emph{monomial}, if $M = S(\sigma)D$ for a permutation $\sigma\in S_k$ and an invertible diagonal matrix $D$.
The decomposition of a monomial matrix into the permutational and the diagonal matrix part is unique.
\end{dfn}

Let $\mathfrak{M} = \mathfrak{M}(k,R,\alpha)$ be the set of all pairs $(N,M)$ of monomial $(k\times k)$-matrices $M$ and $N$ over $R$, such that for each $\alpha$-circulant matrix $A\in R^{k\times k}$, the matrix $N^{-1} A M$ is again $\alpha$-circulant.
An element $(N,M)$ of $\mathfrak{M}$ can be interpreted as a mapping $R^{k\times k}\rightarrow R^{k\times k}$, $A \mapsto N^{-1} A M$.
The composition of mappings implies a group structure on $\mathfrak{M}$, and $\mathfrak{M}$ operates on the set of all $\alpha$-circulant matrices.

Now let $(N,M)\in\mathfrak{M}$.
The codes generated by $(I \mid A)$ and by $(I \mid N^{-1} A M)$ are equivalent, since
\[
N^{-1} (I \mid A) \begin{pmatrix}N & 0 \\ 0 & M\end{pmatrix} = (I \mid N^{-1} A M)
\]
and the matrix $\begin{pmatrix}N & 0 \\ 0 & M\end{pmatrix}$ is monomial.
Thus, $\mathfrak{M}$ also operates on the set of all double $\alpha$-circulant generator matrices.

In general $\mathfrak{M}$-equivalence is weaker than the code equivalence:
For example the vectors $v = (1111101011011010)\in\Z_2^{16}$ and $w = (1110010011100000)\in\Z_2^{16}$ generate two equivalent binary double circulant self-dual $[32,16]$-codes.
But since the number of zeros in $v$ and $w$ is different, the two circulant matrices generated by $v$ and $w$ cannot be in the same $\mathfrak{M}$-orbit.

\section{Monomial transformations of $\alpha$-circulant matrices}
\label{section:monomial}
Let $R$ be a commutative ring, $k\in\N$ and $\alpha\in R$ a unit.
In this section we give some elements $(N,M)$ of the group $\mathfrak{M} = \mathfrak{M}(R,k,\alpha)$ defined in the last section.
In part they can be deduced from \cite[chapter~16, \S6, problem~7]{MacWilliams-Sloane-1977-The_Theory_of_Error_Correcting_Codes}.

Quite obvious elements of $\mathfrak{M}$ are $(I_k, T_\alpha)$, $(T_\alpha, I_k)$, $(I_k, D)$ and $(D, I_k)$, where $D$ denotes an invertible scalar matrix.

For certain $\alpha$ further elements of $\mathfrak{M}$ are given by the following lemma, which is checked by a calculation:
\begin{lma}
Let $\alpha\in R$ with $\alpha^2 = 1$ and $s\in\{0,\ldots,k-1\}$ with $\gcd(s,k) = 1$.
Let $\sigma = (i\mapsto si\mod k)\in S_k$.
We define $D$ as the diagonal matrix which has $\alpha^{(s+1)i + \lfloor si/k\rfloor}$ as $i$-th diagonal entry, and we define the monomial matrix $M = S(\sigma) D$.
Then
\[
    (M,M)\in\mathfrak{M}
\]
More specifically: Let $f\in R[x]/(x^k-\alpha)$.
It holds:\[
M^{-1} \cir_\alpha(f) M = \cir_\alpha(f((\alpha x)^s))
\]
\end{lma}

Finally, there is an invertible transformation $A\mapsto M^{-1} A M$ that converts an $\alpha$-circulant matrix into a $\beta$-circulant matrix for certain pairs $(\alpha,\beta)$:
\begin{lma}
Let $R$ be a commutative ring, $\alpha\in R$ a unit and $\{i,j\}\subset\N$.
Let $A$ be an $\alpha^i$-circulant $(k\times k)$-matrix over $R$ and $M$ the diagonal matrix with the diagonal vector $(1,\alpha^j,\alpha^{2j},\ldots,\alpha^{(k-1)j})$.
Then $M^{-1} A M$ is an $\alpha^{i-kj}$-circulant matrix.
For $\alpha^2 = 1$ the matrix $M$ is orthogonal.
\end{lma}

\section{The lift of an $\alpha$-circulant matrix}
\label{section:lift}

If we want to construct all equivalence classes of double $\alpha$-circulant codes over a commutative ring $R$, it is enough to consider orbit representatives of the group action of $\mathfrak{M}$ on the set of all double $\alpha$-circulant generator matrices, or equivalently, on the set of all $\alpha$-circulant matrices.

Furthermore, we can benefit from non-trivial ideals of $R$:
Let $I$ be an ideal of $R$ with $\{0\}\neq I \neq R$, and $\bar{\ }: R\rightarrow R/I$ the canonical projection of $R$ onto $R/I$.
We set $\mathfrak{M} = \mathfrak{M}(k,R,\alpha)$ and $\bar{\mathfrak{M}} = \{(\bar{N},\bar{M}) : (N,M)\in\mathfrak{M}\}$.
It holds $\bar{\mathfrak{M}}\subseteq\mathfrak{M}(k,R/I,\bar{\alpha})$.
Let $e : R/I \rightarrow R$ be a mapping that maps each element $r + I$ of $R/I$ to a representative element $r\in R$.

\begin{dfn}
Let $A = \cir_{\bar{\alpha}}(v)$ be an $\bar{\alpha}$-circulant matrix with generating vector $v\in R/I$.
An $\alpha$-circulant matrix $B$ over $R$ is called \emph{lift} of $A$, if $\bar{B} = A$.
In this case we also say that the code generated by $(I_k\mid B)$ is a lift of the code generated by $(I_k\mid A)$.
The lifts of $A$ are exactly the matrices of the form $\cir_\alpha(e(v)) + \cir_\alpha(w)$ with $w\in I^k$.\footnote{To avoid confusion, we point out that $I^k$ denotes the $k$-fold Cartesian product $I\times \ldots \times I$ here.}
The vector $w$ is called \emph{lift vector}.
\end{dfn}

To find all double $\alpha$-circulant codes over $R$, we can run over all lifts of all double $\bar{\alpha}$-circulant codes over $R/I$.
The crucial point now is that for finding at least one representative all equivalence classes of double $\alpha$-circulant codes over $R$, it is enough to run over the lifts of a \emph{set of representatives} of the group action of $\bar{\mathfrak{M}}$ on the set of all $\bar{\alpha}$-circulant codes over $R/I$:

\begin{lma}
Let $A$ and $B$ be two $\bar{\alpha}$-circulant matrices over $R/I$ which are in the same $\bar{\mathfrak{M}}$-orbit.
Then for each lift of $A$ there is a lift of $B$ which is in the same $\mathfrak{M}$-orbit.
\end{lma}

\begin{proof}
Because $A$ and $B$ are in the same $\bar{\mathfrak{M}}$-orbit, there is a pair of monomial matrices $(N,M)\in\mathfrak{M}$ such that $\bar{N}^{-1} A \bar{M} = B$.
Let $a\in (R/I)^k$ be the generating vector of $A$ and $b\in (R/I)^k$ the generating vector of $B$.
Since $\overline{\cir_\alpha(e(a))} = A$ and $\overline{\cir_\alpha(e(b))} = B$ it holds $N^{-1} \cir_\alpha(e(a)) M = \cir_\alpha(e(b)) + K$, where $K\in I^{k\times k}$.
$\cir_\alpha(e(b))$ is of course $\alpha$-circulant, and $N^{-1} \cir_\alpha(e(a)) M$ is $\alpha$-circulant because of $(N,M)\in\mathfrak{M}$.
Thus, also $K$ is $\alpha$-circulant and therefore there is a $z\in I^k$ with $\cir_\alpha(z) = K$.

Now, let $w\in I^k$ be some lift vector.
$N^{-1}\cir_\alpha(w) M\in I^{k\times k}$ is $\alpha$-circulant and generated by a lift vector $w^\prime\in I^k$.
Then $N^{-1} (\cir_\alpha(e(a)) + \cir_\alpha(w)) M = \cir_\alpha(e(b)) + \cir_\alpha(z + w^\prime)$, and $z + w^\prime \in I^k$.
Therefore, the lift of $A$ by the lift vector $w$ and the lift of $B$ by the lift vector $z + w^\prime$ are in the same $\mathfrak{M}$-orbit.
\end{proof}



It is not hard to adapt this approach to bordered $\alpha$-circulant codes.
One difference is an additional restriction on the appearing monomial matrices: Its diagonal part must be a scalar matrix. The reason for this is that otherwise the monomial transformations would destroy the border vectors $(\gamma\ldots\gamma)$ and $(\delta\ldots\delta)^t$.

Circulant matrices are often used to construct self-dual codes.
Thus we are interested in a fast way to generate the lifts that lead to self-dual codes.
The next section gives such an algorithm for the case that $R$ is a finite chain ring and $I$ is its minimal ideal.

\section{Self-dual double $\alpha$-circulant codes over finite commutative chain rings}
\label{section:chain_rings}
We want to investigate self-dual double $\alpha$-circulant codes.
Here we need $\alpha^2 = 1$. This is seen by denoting the rows of a generator matrix $G$ of such a code by $w_0\ldots w_{k-1}$, and by comparing the scalar products $\left<w_0,w_1\right>$ and $\left<w_1,w_2\right>$, which must be both zero.
Furthermore, given $\alpha^2 = 1$, we see that $\left<w_0,w_i\right> = \left<w_j,w_{i+j}\right>$, where $i + j$ are taken modulo $k$.
Thus $G$ generates a self-dual code if $\left<w_0,w_0\right> = 1$ and for all $j\in\{1,\ldots,\lfloor k/2\rfloor\}$ the scalar products $\left<w_0,w_j\right>$ are equal to $0$.

\begin{dfn}
A ring $R$ is called \emph{chain ring}, if its left ideals are linearily ordered by inclusion.
\end{dfn}

For the theory of finite chain rings and linear codes over finite chain rings see \cite{Honold-Landjev-2000-EJC7:R11}.

In this section $R$ will be a finite commutative chain ring, which is not a finite field, and $\alpha$ an element of $R$ with $\alpha^2 = 1$.
There is a ring element $\theta\in R$ which generates the maximal ideal $R\theta$ of $R$.
The number $q$ is defined by $R/R\theta \cong\F_q$, and $m$ is defined by $\abs{R} = q^m$.
Because $R$ is not a field, we have $m\geq 2$.
The minimal ideal of $R$ is $R \theta^{m-1}$.
$\mathfrak{M}$ is defined as in section~\ref{section:circ_code}, with with the difference that all monomial matrices $M$ should be orthogonal, that is $M M^t = I_k$.
Thus each $\mathfrak{M}$-image of a generator matrix of a self-dual code again generates a self-dual code.

Now let $I = R \theta^{m-1}$ be the minimal ideal of $R$.
As in section~\ref{section:lift} let $e : R/I \rightarrow R$ be a mapping that assignes each element of $R/I$ to a representative in $R$, now with the additional condition $e(\bar{\alpha}) = \alpha$.

We mention that if $(I_k \mid B)$ generates a double $\alpha$-circulant self-dual code over $R$, then $(I_k \mid \bar{B})$ generates a double $\bar{\alpha}$-circulant self-dual code over $R/I$.
So $B$ is among the lifts of all $\bar{\alpha}$-circulant matrices $A$ over $R/I$ such that $(I_k \mid A)$ generates a self-dual double $\bar{\alpha}$-circulant code.

Let $A = \cir_{\bar{\alpha}}(a)$ be an $\bar{\alpha}$-circulant matrix over $R/I$ such that $(I_k\mid A)$ generates a self-dual code.
So $A A^t = -I_k$, and therefore
\begin{align*}
 c_0 & := 1 + \sum_{i=0}^{k-1}e(a_i)^2 \in I\quad\mbox{and} \\
 c_j & := \sum_{i=0}^{j-1}\alpha e(a_i) e(a_{k-j+i}) + \sum_{i=j}^{k-1}e(a_i)e(a_{i-j}) \in I\quad\mbox{for all }j\in\{1,\ldots,\lfloor k/2\rfloor\}
\end{align*}
We want to find all lifts $B = \cir_\alpha(e(a)) + \cir_\alpha(w)$ of $A$ with $w\in I^k$ such that $B B^t = -I_k$.
As we have seen, this is equivalent to
\begin{align*}
 0 & = 1 + \sum_{i=0}^{k-1}(e(a_i) + w_i)^2\quad\mbox{and} \\
 0 & = \sum_{i=0}^{j-1}(e(a_i) + w_i)(\alpha e(a_{k-j+i}) + w_{k-j+i}) + \sum_{i=j}^{k-1}(e(a_i) + w_i)(e(a_{i-j})+ w_{i-j})
\end{align*}
where the second equation holds for all $j\in\{1,\ldots,\lfloor k/2\rfloor\}$.
Using $I\cdot I = 0$, we get
\begin{align*}
0 & = c_0 + 2\sum_{i=0}^{k-1}e(a_i) w_i\quad\mbox{and} \\
0 & = c_j + \sum_{i=0}^{j-1}(e(a_i) w_{k-j+i} + \alpha e(a_{k-j+i}) w_i) + \sum_{i=j}^{k-1}(e(a_i) w_{i-j} + e(a_{i-j}) w_i)
\end{align*}
This is a $R$-linear system of equations for the components $w_i\in I$ of the lift vector.
Using the fact that the $R$-modules $R/(R \theta)$ and $I$ are isomorphic, and $R/(R \theta)\cong\F_q$, this can be reformulated as a linear system of equations over the finite field $\F_q$, which can be solved efficiently.

Since $R/I$ is again a commutative chain ring, the lifting step can be applied repeatedly. Thus, starting with the codes over $\F_q$, the codes over $R$ can be constructed by $m-1$ nested lifting steps.

Again, this method can be adapted to bordered $\alpha$-circulant matrices over commutative finite chain rings.

\section{Application: Self-dual codes over $\mathbb{Z}_4$}
For a fixed length $n$ we want to find the highest minimum Lee distance $d_\Lee$ of double nega-circulant and bordered circulant self-dual codes over $\mathbb{Z}_4$.
In \cite{Gulliver-Harada-1999-DM194:129-137} codes of the bordered circulant type of length up to $32$ were investigated.

First we notice that the length $n$ must be a multiple of $8$:
Let $C$ be a bordered circulant or a double nega-circulant code of length $n$ and $c$ a codeword of $C$.
We have $0 = \left<c,c\right> = \sum_{i=0}^{n-1} c_i^2\in\Z_4$.
The last expression equals the number of units in $c$ modulo 4, so the number of units of each codeword is a multiple of $4$.
It follows that the image $\bar{C}$ of $C$ over $\Z_2$ is a doubly-even self-dual code of length $n$, which can only exist for lengths $n$ divisible by $8$.

Furthermore, it holds
\begin{equation}
\label{equ:mindist}
d_\Lee(C) \leq 2 d_\Ham(\bar{C})
\end{equation}
As a result, we only need to consider the lifts of codes $\bar{C}$ which have a sufficiently high minimum Hamming distance.

We explain the algorithm for the case of the nega-circulant codes:
In a first step, for a given length $n$ we generate all doubly-even double circulant self-dual codes over $\Z_2$.
This is done by enumerating Lyndon words of length $n$ which serve as generating vectors for the circulant matrix.
Next, we filter out all duplicates with respect to the group action of $\overline{\mathfrak{M}}$, where $\mathfrak{M}$ is the group generated by the elements given in section~\ref{section:monomial} which consist of pairs of orthogonal monomial matrices.

A variable $d$ will keep the best minimum Lee distance we already found.
We initialize $d$ with $0$.
Now we loop over all binary codes $C_{\Z_2}$ in our list, from the higher to the lower minimum Hamming distance of $C_{\Z_2}$:
If $2 d_\Ham(C_{\Z_2}) \leq d$ we are finished because of (\ref{equ:mindist}).
Otherwise, as explained in section~\ref{section:chain_rings}, we solve a system of linear equations over $\Z_2$ and get all self-dual lifts of $C_{\Z_2}$.
For these lifts we compute the minimum Lee distance and update $d$ accordingly.

Most of the computation time is spent on the computation of the minimum Lee distances.
Thus it was a crucial point to write a specialized algorithm for this purpose.
It is described in \cite{Kiermaier-Wassermann-2008}.

The results of our search are displayed in the following table.
For given length $n$, it lists the highest minimum Lee distance of a self-dual code of the respective type:
\[
\begin{array}{r|rrrrrrrr}
n                                     & 8 & 16 & 24 & 32 & 40 & 48 & 56 & 64 \\
\hline
\mbox{double nega-circulant} & 6 &  8 & 12 & 14 & 14 & 18 & 16 & 20 \\
\mbox{bordered circulant}    & 6 &  8 & 12 & 14 & 14 & 18 & 18 & 20
\end{array}
\]
We see that the results are identical for the two classes of codes, except for length $56$.
Using (\ref{equ:mindist}) there is a simple reason that for this length no double circulant self-dual code over $\Z_4$ with minimum Lee distance greater than $16$ exists: The best doubly-even double circulant self-dual binary code has only minimum Hamming distance $8$.

\section*{Acknowledgment}
This research was supported in part by Deutsche Forschungsgemeinschaft \textbf{WA 1666/4-1}.


\end{document}